\documentclass[]{article}
\usepackage{amssymb,amsmath,amsthm,mathrsfs,nicefrac,url,color}

\usepackage{pdfsync}

\newcommand{\R}{\mathbf{R}}

\renewcommand{\P}{\mathrm{P}}
\newcommand {\E}{\mathrm{E}}

\renewcommand{\d}{\text{\rm d}}

\newcommand{\sG}{\mathcal{G}}

\newcommand{\sE}{\mathcal{E}}
\newcommand{\sI}{\mathcal{I}}
\newcommand{\sS}{\mathcal{S}}

\newcommand{\ve}{\epsilon}

\newtheorem{stat}{Statement}[section]
\newtheorem{proposition}[stat]{Proposition}
\newtheorem{corollary}[stat]{Corollary}
\newtheorem{theorem}[stat]{Theorem}
\newtheorem{lemma}[stat]{Lemma}

\theoremstyle{definition}
\newtheorem{definition}[stat]{Definition}\newtheorem{remark}[stat]{Remark}

\numberwithin{equation}{section}

\begin{document}

\title{\bf Remarks on non-linear noise excitability of some stochastic heat equations.%
	\thanks{%
	Research supported in part by EPSRC.}}
	
\author{Mohammud Foondun\\Loughborough University
\and Mathew Joseph\\University of Sheffield}

\date{January 31, 2014}
\maketitle
\begin{abstract}
	We consider nonlinear parabolic
	SPDEs of the form
	$\partial_t u=\Delta u + \lambda \sigma(u)\dot w$ on the interval $(0, L)$, where
	$\dot w$ denotes space-time white noise,
	$\sigma$ is Lipschitz continuous. Under Dirichlet boundary conditions and a linear growth condition on $\sigma$,  we show that
	the expected $L^2$-energy is of order $\exp[\text{const}\times\lambda^4]$ as $\lambda\rightarrow \infty$.
	This significantly improves a recent result of Khoshnevisan and  Kim \cite{KK1}.  Our method is very different from theirs and it allows us to arrive at the same conclusion for the same equation but with Neumann boundary condition.
	 This improves over another result in \cite{KK1}.\\
	
		
		
	\noindent{\it Running Title:} Noise excitability and
		parabolic SPDEs.\newpage
\end{abstract}

\section{Introduction}

The main objective of this paper is to study the effect of noise on the solutions to various stochastic heat equations.  Fix $L>0$ and consider the following
\begin{equation}\label{eq:dirichlet}
\left|\begin{split}
&\partial_t u_t(x)=\frac{1}{2}\Delta u_t(x)+\lambda \sigma(u_t(x))\dot{w}(t,\,x),\\
&u_t(0)=0, \quad u_t(L)=0,
\end{split}
\right.
\end{equation}
where $\dot{w}$ denotes the space time white noise on $(0,\infty)\times (0,\,L)$. $\lambda$ is a positive number called the {\it level of the noise}.  Here and throughout this paper, the initial function $u_0:[0,\,L]\rightarrow \R_+$ is a nonrandom function which is compactly supported in $[0,\,L]$. $\sigma:\R \rightarrow \R$ is a continuous function with $\sigma(0)=0$ and
\begin{equation*}
l_\sigma:=\inf_{x\in \R\backslash \{0\}}\left| \frac{\sigma(x)}{x}\right|\quad \text{and}\quad
L_\sigma:=\sup_{x\in \R\backslash \{0\}}\left| \frac{\sigma(x)}{x}\right|,
\end{equation*}
where $0<l_\sigma\leq L_\sigma<\infty$.

Our study is motivated by a recent paper of Khoshnevisan and Kim \cite{KK1} where the authors initiated the study of the effect of $\lambda$ on the energy of the solution. We will shortly describe their results in a bit more detail but let us mention that existence and uniqueness is not an issue for us.  It is well known that the above equation has a unique mild solution satisfying

\begin{equation}\label{moments}
\sup_{x\in[0,\,L]}\sup_{t\in[0,\,T]} \E|u_t(x)|^k<\infty \quad\text{for all}\quad T>0 \quad\text{and}\quad k\in[2,\,\infty].
\end{equation}

Moreover the solution admits the following integral representation
\begin{equation}\label{mild:dirichlet}
u_t(x)=
(\sG_Du)_t(x)+ \lambda \int_0^L\int_0^t p_D(t-s,\,x,\,y)\sigma(u_s(y))w(\d s\,\d y),
\end{equation}
where
\begin{equation*}
(\sG_D u)_t(x):=\int_0^L u_0(y)p_D(t,\,x,\,y)\,\d y,
\end{equation*}
and $p_D(t,\,x,\,y)$ denotes the Dirichlet heat kernel. As usual, \eqref{mild:dirichlet} will be the starting point of most of our analysis.  For more information about existence-uniquness , see \cite{Minicourse} or \cite{Walsh} for more information.

To describe our results in a precise manner, we adopt some notations and definitions from \cite{KK1} and \cite{KK2}.  We begin by defining  the {\it energy of the solution at time t} by
\begin{equation}\label{energy}
\sE_t(\lambda):=\sqrt{\E\left(\|u_t\|^2_{L^2[0,\,L]}\right)}.
\end{equation}

One of the main results in \cite{KK1} states that as $\lambda$ gets large, $\sE_t(\lambda)$ grows at most like $\exp(\text{const}\times \lambda^4)$ but at least like $\exp(\text{const}\times \lambda^2)$. This current project grew out of trying to understand this discrepancy. The following indices were introduced in \cite{KK2} to capture the {\it super exponential} growth just mentioned.

\begin{definition}
The {\it upper excitation index} of $u$ at time $t$ is given by
\begin{equation*}
\overline{e}(t):=\limsup_{\lambda\rightarrow \infty}\frac{\log \log \sE_t(\lambda)}{\log \lambda}
\end{equation*}
\end{definition}

\begin{definition}
The {\it lower excitation index} of $u$ at time $t$ is given by
\begin{equation*}
\underline{e}(t):=\liminf_{\lambda\rightarrow \infty}\frac{\log \log \sE_t(\lambda)}{\log \lambda}
\end{equation*}
\end{definition}

When $\bar{e}(t)$ and $\underline{e}(t)$ are equal, we simply refer to the common value as the {\it nonlinear noise excitation index} of the solution at time $t$, which is required to be strictly positive. We are now ready to state the first main result of the paper.

\begin{theorem}\label{theo:dirich:energy}
The nonlinear excitation index of the solution to \eqref{eq:dirichlet} is $4$.
\end{theorem}


Estimating the lower excitation index is the main contribution of this paper and our approach requires two new ideas which we now describe.

\begin{enumerate}
\item We use a couple of renewal inequalities which give the desired upper and lower bound on the energy.  The use of renewal theoretic ideas was introduced in \cite{FK} but here we use it in a different manner; see Remark \ref{renew}.
\item To arrive at these renewal inequalities, we make use of the idea that for small times and away from the boundary, the Dirichlet heat kernel behaves pretty much like that Gaussian heat kernel.  This idea has been the subject of intense investigations for decades now; see \cite{Bass} and \cite{Stroock}. Since we are working in spatial dimension one, we provide complete analytic proofs of the main estimates we need.
\end{enumerate}
It is also interesting to note that in \cite{KK1}, the bound on the upper index was the harder part of the proof. Here the complete opposite is true; the lower bound is much harder and requires the second point mentioned above which is entirely novel. As far as we know, Gaussian estimates for Dirichlet Laplacian have never been used in the study of stochastic partial differential differential equations.  Using these two ideas, we were able to improve the bound on the lower index.

It turns out that our method can adapted  to study the same stochastic PDE but with Neumann boundary condition.  We now describe our main findings in this context.  Consider the following equation

\begin{equation}\label{eq:neumann}
\left|\begin{split}
&\partial_t u_t(x)=\frac{1}{2}\Delta u_t(x)+\lambda \sigma(u_t(x))\dot{w}(t,\,x),\\
&\partial_x u_t(0)=0, \quad \partial_xu_t(L)=0.
\end{split}
\right.
\end{equation}

Here we stress the fact that as opposed to \cite{KK1}, we do not require our initial function to be bounded below.  Any bounded nonrandom compactly supported non negative initial function will be enough.  It is well known that there exists a unique mild solution satisfying \eqref{moments} with the following integral representation,

\begin{equation}\label{mild:neumann}
u_t(x)=
(\sG_Nu)_t(x)+ \lambda \int_0^L\int_0^t p_N(t-s,\,x,\,y)\sigma(u_s(y))w(\d s\,\d y),
\end{equation}
where
\begin{equation*}
(\sG_N u)_t(x):=\int_0^L u_0(y)p_N(t,\,x,\,y)\,\d y
\end{equation*}
and $p_N(t,\,x,\,y)$ is the Neumann heat kernel.  We again refer to \cite{Minicourse} and \cite{Walsh} for more information about various technicalities.  To state our main result for \eqref{eq:neumann}, we set the following notations,
\begin{equation}\label{inf}
\sI_t(\lambda):=\inf_{x\in [0,\,L]}\E|u_t(x)|^2
\end{equation}
and
\begin{equation}\label{sup}
\sS_t(\lambda):=\sup_{x\in [0,\,L]}\E|u_t(x)|^2,
\end{equation}
where $u_t$ is the solution to \eqref{eq:neumann}.
\begin{theorem}\label{theo:neu}
Fix $t>0$, then
\begin{equation*}
\liminf_{\lambda\rightarrow \infty} \frac{\log\log \sI_t(\lambda)}{\log \lambda} =\limsup_{\lambda\rightarrow \infty} \frac{\log\log \sS_t(\lambda)}{\log \lambda}=4.
\end{equation*}
\end{theorem}
An immediate consequence of the above is the following.
\begin{corollary}\label{corr:energy}
The nonlinear excitation index of the solution to \eqref{eq:neumann} is $4$.
\end{corollary}

Our technique seems to be suited for the study of a wider class of the stochastic equations.  If the Laplacian in say \eqref{eq:dirichlet} were replaced by the fractional Dirichlet Laplacian of order $\alpha$ and the white noise were replaced by a colored noise with Riesz Kernel of order $\beta$, we conjecture that the non-linear excitation index is $2\alpha/(\alpha-\beta)$. This is currently under investigation and will be the subject of \cite{FT}. 

We end this introduction with the plan of the article.  Section 2 contains the renewal type inequalities.  Section 3 contains the relevant Dirichlet heat estimates and the proof of Theorem \ref{theo:dirich:energy}. Section 4 contains the corresponding estimates for the Neuman heat kernel as well as the  proof of Theorem \ref{theo:neu} and its corollary.

\section{Some estimates}
This section will be devoted to the renewal-type inequalities mentioned in the introduction.  The perceptive reader will recognise  that the presence of the square root inside the integrals is motivated by the Gaussian heat kernel.

\begin{proposition}\label{upperbound}
Suppose that $f(t),$ is a non-negative integrable function on $0\leq t\leq T$ satisfying

\begin{equation}\label{upper}
f(t)\leq a+bk\int_0^t\frac{f(s)}{\sqrt{t-s}}\,\d s\quad \text{for\,all} \quad k>0\quad \text{and}\quad 0\leq t\leq T,
\end{equation}
where $a$ and $b$ are positive constants and $T< \infty$. Then for each $0\le t\le T$, we have
\begin{equation*}
\limsup_{k\rightarrow \infty}\frac{\log \log f(t)}{\log k}\leq 2.
\end{equation*}

\end{proposition}

\begin{proof}
We start off by iterating inequality \eqref{upper} once to obtain \begin{eqnarray*}
f(t)&\leq& a+bk\int_0^t\frac{f(s)}{\sqrt{t-s}}\,\d s\\
&=&a+abk \sqrt{t}+b^2k^2\int_0^t\int_0^s\frac{f(l)}{\sqrt{(t-s)(s-l)}}\d l\d s.
\end{eqnarray*}
We change the order of integration in the above double integral to find that
\begin{eqnarray*}
\int_0^t\int_0^s\frac{f(l)}{\sqrt{(t-s)(s-l)}}\d l\d s\hskip1.0in\\
=\int_0^t\int_l^t\frac{f(l)}{\sqrt{(t-s)(s-l)}}\d s\d l\\
=c_1\int_0^tf(l)\d l,\hskip0.75in
\end{eqnarray*}
where $c_1$ is some positive constant. This together with the above inequality gives

\begin{equation*}
f(t)\leq c_2+c_2k\sqrt{t}+c_2k^2\int_0^tf(l)\d l,
\end{equation*}
for some positive constant $c_2$.  A suitable version of Gronwall's inequality now finishes the proof of the proposition.
\end{proof}

We now reverse the inequality in the statement of the proposition to obtain the following result.  The proof is pretty much the same as the above, so we omit it.

\begin{proposition}\label{lowerbound}
Suppose  that $f(t)$ is a non-negative integrable functions on $0\leq t< T$ satisfying
\begin{equation}\label{lower}
f(t)\geq a+bk\int_0^t\frac{f(s)}{\sqrt{t-s}}\,\d s\quad \text{for\,all} \quad k>0\quad \text{and}\quad 0\leq t\leq T,
\end{equation}
where $a$ and $b$ are positive constants and $T< \infty$. Then for each $0\le t\le T$, we have
\begin{equation*}
\liminf_{k\rightarrow \infty}\frac{\log \log f(t)}{\log k}\geq 2.
\end{equation*}
\end{proposition}

\begin{remark}\label{renew}
Obviously, the constants $a$ and $b$ appearing in Proposition \ref{upperbound} might be different from those appearing in Proposition \ref{lowerbound}.  In a sense, the above results are not new. But what's original about them here, is that they are used to get information about the growth of the function with respect to the parameter $k$ rather than $t$.
\end{remark}

\section{The Dirichlet equation.}

We start off with a result which gives a lower bound on the Dirichlet heat kernel in terms of the Gaussian heat kernel.  This is borrowed from \cite{Berg}. But we give a proof here for the sake of completeness.  Recall that from the method of images, we have the following representation,

\begin{equation}\label{images-1}
p_D(t,\,x,\,y)=\frac{1}{\sqrt{4\pi t}}\sum_{n=-\infty}^{\infty}\Big[e^{-\frac{|x-(y+2nL)|^2}{4t}}-e^{-\frac{|x-(-y+2nL)|^2}{4t}}\Big].
\end{equation}

\begin{lemma}
Suppose that $x,\,y\in (0, L)$ and set $\epsilon:=\min\{x, y, L-x, L-y\}$, then we have
\begin{equation*}
p_D(t,\,x,\,y)\geq (1-2e^{-\epsilon^2/t})p(t,\,x,\,y).
\end{equation*}
\end{lemma}

\begin{proof}
The proof involves rewriting \eqref{images-1} in a suitable way and making use of  the following observation. For $n\geq 1$ and $x, y\in(0,L)$

\begin{equation}\label{ob-1}
|x+y+2nL|\geq|x-y+2nL|
\end{equation}
and

\begin{equation}\label{ob-2}
|-(x+y)+2(n+1)L|\geq|2nL-(x-y)|.
\end{equation}
We can now write
\begin{equation*}
\begin{aligned}
\sum_{n=-\infty}^{\infty}\Big[e^{-\frac{|x-(y+2nL)|^2}{4t}}&-e^{-\frac{|x-(-y+2nL)|^2}{4t}}\Big]\\
&=e^{-\frac{|x-y|^2}{4t}}-e^{-\frac{|x+y|^2}{4t}}-e^{-\frac{|x+y-2L|^2}{4t}}\\
&\;\;+\sum_{n=1}^\infty\Big[e^{-\frac{|x-y-2nL|^2}{4t}}+e^{-\frac{|x-y+2nL|^2}{4t}}\\
&\qquad \qquad -e^{-\frac{|x+y-2(n+1)L)|^2}{4t}}-e^{-\frac{|x+y+2nL|^2}{4t}}\Big].
\end{aligned}
\end{equation*}

We now use \eqref{ob-1} and \eqref{ob-2} together with \eqref{images-1} to conclude that

\begin{equation*}
\begin{aligned}
p_D(t,\,x,\,y)&\geq\frac{1}{\sqrt{4\pi t}}\left[e^{-\frac{|x-y|^2}{4t}}-e^{-\frac{|x+y|^2}{4t}}-e^{-\frac{|x+y-2L|^2}{4t}}\right]\\
&\geq\frac{1}{\sqrt{4\pi t}}e^{-\frac{|x-y|^2}{4t}}\left[1-e^{-\frac{xy}{t}}-e^{-\frac{(L-x)(L-y)}{t}}\right].
\end{aligned}
\end{equation*}
This and the definition of $\ve$ essentially finish the proof.
\end{proof}

A consequence of the above lemma is that away from the boundary, the Dirichlet heat kernel behaves pretty much like the Gaussian one provided that time is small enough.  This is intuitively obvious from the probabilistic point of view.

\begin{corollary}\label{cor:comp-dirich}
Fix $\epsilon>0$, then there exists $t_0>0$ depending on $\epsilon$ such that for all $t\leq t_0$ and all $x,\,y\in [\epsilon, L-\epsilon]$, we have
\begin{equation*}
p_D(t,\,x,\,y)\geq \frac{1}{2}p(t,\,x,\,y).
\end{equation*}

\end{corollary}
\begin{proof}
Fix $\epsilon>0$. For $t\leq \frac{\epsilon^2}{\ln 4}$, we have $1-2e^{-\epsilon^2/t}\geq \frac{1}{2}$.  The result then follows from the above lemma.
\end{proof}

Another starting point for the proof of the above result could be the following. Recall that the Dirichlet heat kernel is the transition probability for a killed Brownian motion. Let $\tau$ denote the first exit time of Brownian motion from from $(0,\,L)$, then since $p(t,\,x,\,y)$ is the transition probability of this Brownian motion, we have
\begin{equation*}
p_D(t,\,x,\,y)=\P^x(\tau>t|B_t=y)p(t,\,x,\,y).
\end{equation*}
We also have
\begin{equation*}
p_D(t,\,x,\,y)=p(t,\,x,\,y)-\E^x(p(t-\tau,\,X_D,\,y); \tau<t),
\end{equation*}
where $X_D$ is the position of the Brownian motion when it hits the boundary, making the following trivial.
\begin{lemma}
For all $x,\,y\in (0,\,L)$ and $t>0$, the following holds
\begin{equation}\label{dirichupper}
p_D(t,\,x,\,y)\leq p(t,\,x,\,y).
\end{equation}
\end{lemma}
The next result says that provided we stay from the boundary, we can find a suitable lower bound on the growth of the second moment of the solution of \eqref{eq:dirichlet}. To state our result, we introduce the following notation. For $\epsilon>0$,

\begin{equation}\label{inf-epsilon}
\sI_{\epsilon,\,t}(\lambda):=\inf_{x\in [\epsilon,\,L-\epsilon]}\E|u_t(x)|^2,
\end{equation}
where $u_t$ is the solution to \eqref{eq:dirichlet}.

\begin{proposition}\label{lower-point}
Fix $\epsilon>0$. Then there exists $t_0>0$ such that for all $t\leq t_0$, we have
\begin{equation*}
\liminf_{\lambda \rightarrow \infty}\frac{\log \log \sI_{\epsilon,\,t}(\lambda)}{\log \lambda}\geq 4.
\end{equation*}
\end{proposition}

\begin{proof}
Using the mild formulation and Ito's isometry, we have
\begin{equation}\label{eq:iso-dirich}
\E|u_t(x)|^2=|(\sG_Du)_t(x)|^2+\lambda^2 \int_0^L\int_0^t p^2_D(t-s,\,x,\,y)\E|\sigma(u_s(y))|^2\d s\,\d y
\end{equation}

We now fix an $\ve>0$ and let $t_0$ be defined as in the proof of Corollary \ref{cor:comp-dirich}. We bound the first term on the right hand side of the above display first. Recall that $(\sG_Du)_t(x)$ solves the deterministic heat equation, that is \eqref{eq:dirichlet} with $\lambda=0$.  Provided we stay away from the boundary, it is bounded below by a constant which depends on $t$.  In other words for $x\in [\ve, L-\ve]$, $|(\sG_Du)_t(x)|^2\geq c_1$ for some positive constant $c_1$ depending on $t$.  We now look at the second term. Using Corollary \ref{cor:comp-dirich}, we obtain
\begin{equation*}
\begin{aligned}
\lambda^2 \int_0^L\int_0^t &p^2_D(t-s,\,x,\,y)\E|\sigma(u_s(y))|^2\d s\,\d y\hskip1.0in\\
&\geq\frac{\lambda^2l_\sigma^2}{4} \int_\ve^{L-\ve}\int_0^t p^2(t-s,\,x,\,y)\E|(u_s(y))|^2\d s\,\d y\\
&\geq\frac{\lambda^2l_\sigma^2}{4}\int_0^t \sI_{\ve,\,t}(\lambda)\int_\ve^{L-\ve}p^2(t-s,\,x,\,y)\,\d y\,\d s.
\end{aligned}
\end{equation*}
We now estimate the innermost integral appearing in the above line. For fixed $t$, $s$ and $x\in [\ve, L-\ve]$, set $D:=[\ve,\, L-\ve]\cap\{y: |y-x|\leq \sqrt{t-s}\}$. Hence for $y\in D$, we have $p(t-s,\,x,\,y)\geq c_2/\sqrt{t-s}$, for some constant $c_2$. We therefore have
\begin{equation*}
\label{innerintegral}
\begin{aligned}
\int_\ve^{L-\ve}p^2(t-s,\,x,\,y)\,\d y&\geq c_3\int_D\frac{1}{t-s} \d y\nonumber\\
&\geq \frac{c_4}{\sqrt{t-s}},
\end{aligned}
\end{equation*}
for some constants $c_3$ and $c_4$. We thus have

\begin{equation*}
\begin{aligned}
\lambda^2 \int_0^L\int_0^t &p^2_D(t-s,\,x,\,y)\E|\sigma(u_s(y))|^2\d s\,\d y\hskip1.0in\\
&\geq \lambda^2l_\sigma^2c_4\int_0^t\frac{\sI_{\epsilon,\,s}(\lambda)}{\sqrt{t-s}}\,\d s.
\end{aligned}
\end{equation*}
We now combine the above estimates yield the following inequality

\begin{equation}\label{I:lowerbound}
\sI_{\epsilon,\, t}(\lambda)\geq c_5+\lambda^2l_\sigma^2c_4\int_0^t\frac{\sI_{\epsilon,\,s}(\lambda)}{\sqrt{t-s}}\,\d s.
\end{equation}

The proof now follows from an application of Proposition \ref{lowerbound}.
\end{proof}
We are now ready to prove Theorem \ref{theo:dirich:energy}.
\subsection{Proof of Theorem \ref{theo:dirich:energy}.}
We will first show that $\overline{e}(t)\leq 4$. This will be done in one step.  We will then show that $\underline{e}(t)\geq 4$. We will do so in two steps.  We prove the bound for small times and then extend it to all times by using a suitable trick.

{\it Proof of the upper bound.}

We start off with the mild formulation, take the second moment and then integrate to obtain

\begin{equation}\label{averagemild}
\begin{aligned}
\int_0^L\E|u_t(x)|^2\,\d x&=\int_0^L|(\sG_Du)_t(x)|^2\,\d x\\
&+\lambda^2 \int_0^L\int_0^L\int_0^t p^2_D(t-s,\,x,\,y)\E|\sigma(u_s(y))|^2\d s\,\d y\,\d x.
\end{aligned}
\end{equation}
For fixed $t$, the first term is a bounded function, so that we have $\int_0^L|(\sG_Du)_t(x)|^2\,\d x\leq c_1$.

We now turn our attention to the second term. 
Using \eqref{dirichupper} and the semigroup property of the heat kernel, we end up with
\begin{equation*}
\begin{aligned}
\lambda^2 \int_0^L\int_0^L\int_0^t &p^2_D(t-s,\,x,\,y)\E|\sigma(u_s(y))|^2\d s\,\d y\,\d x\hskip1.0in\\
&\leq\lambda^2L_\sigma^2 \int_0^L\int_0^t p(2(t-s),\,y,\,y)\E|u_s(y)|^2\d s\,\d y\\
&\leq c_3\lambda^2L_\sigma^2\int_0^t\frac{\sE_s^2(\lambda)}{\sqrt{t-s}}\,\d s.
\end{aligned}
\end{equation*}
We now combine the above estimates to obtain
\begin{equation*}
\sE^2_t(\lambda)\leq c_4+c_4\lambda^2L_\sigma^2\int_0^t\frac{\sE_s^2(\lambda)}{\sqrt{t-s}}\,\d s.
\end{equation*}
The upper bound is thus proved after an application of Proposition \ref{upperbound}.

{\it Proof of the lower bound.}\\
{\it Step 1:}  We first prove the lower bound for $t\leq t_0$ where $t_0$ is some positive number.  Being the solution to the deterministic heat equation, $(\sG_Du)_t(x)$ is bounded below by a constant depending on $t$. So the first term of \eqref{averagemild} is thus bounded below.

To find a lower bound on \eqref{averagemild}, we use the lower bound on $\sigma$ as well as the Markov property of killed Brownian motion to find that
\begin{equation*}
\begin{aligned}
\lambda^2 \int_0^L\int_0^L\int_0^t &p^2_D(t-s,\,x,\,y)\E|\sigma(u_s(y))|^2\d s\,\d y\,\d x\hskip1.0in\\
&\geq \lambda^2l_\sigma^2 \int_0^L\int_0^L\int_0^t p^2_D(t-s,\,x,\,y)\E|u_s(y)|^2\d s\,\d x\,\d y\\
&= \lambda^2l_\sigma^2 \int_0^L\int_0^t p_D(2(t-s),\,y,\,y)\E|u_s(y)|^2\d s\,\d y.
\end{aligned}
\end{equation*}
As in the proof of Proposition \ref{lower-point}, we have that the above is
\begin{equation*}
\begin{aligned}
&\geq \lambda^2l_\sigma^2 \int_{\ve}^{L-\ve}\int_0^t p_D(2(t-s),\,y,\,y)\E|u_s(y)|^2\d s\,\d y\\
&\geq \lambda^2l_\sigma^2 \int_0^t\sI_{\ve,\,s}(\lambda)\int_{\ve}^{L-\ve} p_D(2(t-s),\,y,\,y)\d y\,\d s\\
&\geq\lambda^2l_\sigma^2\int_0^{t/2}\sI_{\ve,\,s}(\lambda)\int_{\ve}^{L-\ve} p_D(2(t-s),\,y,\,y)\d y\,\d s.
\end{aligned}
\end{equation*}
The next step is to bound the inner integral appearing in the last display. Corollary \ref{cor:comp-dirich} shows that there exists a constant $c_5$ such that
\begin{equation*}
\begin{aligned}
\int_\ve^{L-\ve} p_D(2(t-s),\,y,\,y)\d y&\geq \frac{1}{2}\int_\ve^{L-\ve} p(2(t-s),\,y,\,y)\d y\\
&\geq c_5.
\end{aligned}
\end{equation*}
if $t\le t_0$ where $t_0$ depends on $\epsilon$. We combine the above estimates to obtain

\begin{equation*}
\int_0^L\E|u_t(x)|^2\,\d x\geq c_6+\lambda^2l_\sigma^2c_7 \int_0^{t/2}\sI_{\ve,\,s}(\lambda)\d s.
\end{equation*}
We now note that \eqref{I:lowerbound} actually means that $\sI_{\ve,\,s}(\lambda)$ grows at least like $\exp(c_8\lambda^4)$. Some calculus then finishes the proof for $t\leq t_0$.

{\it Step 2:}  We now show that the lower bound holds for any $t>0$.  We assume that $t>t_0$, otherwise there is nothing else to prove.  Let $t_1$ be a small constant to be chosen later.  We write $t=t-t_1+t_1$ and set $T=t-t_1$ for notational convenience.  As we have seen before the mild formulation of the solution yields
\begin{equation*}
\E|u_t(x)|^2=|(\sG_Du)_t(x)|^2+\lambda^2 \int_0^L\int_0^t p^2_D(t-s,\,x,\,y)\E|\sigma(u_s(y))|^2\d s\,\d y.
\end{equation*}
We now use the notation $t=T+t_1$ in the above display together with a few lines of computations to obtain

\begin{equation*}
\begin{aligned}
\E|u_{T+t_1}&(x)|^2\\
&=|(\sG_Du)_{T+t_1}(x)|^2+\lambda^2 \int_0^L\int_0^T p^2_D(T+t_1-s,\,x,\,y)\E|\sigma(u_s(y))|^2\d s\,\d y\\
&+\lambda^2 \int_0^L\int_T^{T+t_1} p^2_D(T+t_1-s,\,x,\,y)\E|\sigma(u_s(y))|^2\d s\,\d y,
\end{aligned}
\end{equation*}
which after a change of variable reduces to

\begin{equation*}
\begin{aligned}
\E|u_{T+t_1}&(x)|^2\\
&=|(\sG_Du)_{T+t_1}(x)|^2+\lambda^2 \int_0^L\int_0^T p^2_D(T+t_1-s,\,x,\,y)\E|\sigma(u_s(y))|^2\d s\,\d y\\
&+\lambda^2 \int_0^L\int_0^{t_1} p^2_D(t_1-s,\,x,\,y)\E|\sigma(u_{T+s}(y))|^2\d s\,\d y.
\end{aligned}
\end{equation*}

We now set \begin{equation*}v_s(x):=u_{T+s}(x)\end{equation*} to see that the above gives the following inequality

\begin{equation*}
\begin{aligned}
\E|v_{t_1}&(x)|^2\\
&\geq |(\sG_Du)_{T+t_1}(x)|^2 +\lambda^2 \int_0^L\int_0^{t_1} p^2_D(t_1-s,\,x,\,y)\E|\sigma(v_s(y))|^2\d s\,\d y.
\end{aligned}
\end{equation*}

A close inspection at the proof given in step 1 shows that the above inequality is all that we need once we choose $t_1=t_0/2$ and show that $|(\sG_Du)_{T+t_1}(x)|^2$ is strictly positive.  But the latter fact can be shown to be bounded below by a positive constant which depends on $T$ and $t_1$.  We leave it to the reader to fill in the details.
\qed

\section{The Neumann equation}

We begin this section with a couple of estimates on the Neumann heat kernel.  First, recall that by the method of images, we have
\begin{equation}\label{images-2}
p_N(t,\,x,\,y)=\frac{1}{\sqrt{4\pi t}}\sum_{n=-\infty}^\infty[e^{-\frac{|x-y-2nL|^2}{4t}}+e^{-\frac{|x+y-2nL|^2}{4t}}].
\end{equation}

From the above series, we trivially have
\begin{lemma}
For all $x,\,y\in [0,\,L]$ and $t>0$, the following holds
\begin{equation}\label{neulower}
p_N(t,\,x,\,y)\geq p(t,\,x,\,y).
\end{equation}
\end{lemma}

And a little more work shows the following,
\begin{lemma}
Let $T>0$, then for all $x,\,y\in [0,\,L]$ and $t\leq T$, the following holds
\begin{equation}\label{neuupper}
p_N(t,\,x,\,y)\leq c_Tp(t,\,x,\,y),
\end{equation}
where $c_T$ is some constant depending on $T>0$.
\end{lemma}
\subsection{Proof of Theorem \ref{theo:neu}}
\begin{proof}
We begin by proving the lower bound first. We start off with the mild formulation and take second moment to end up with
\begin{equation}\label{eq:iso-neu}
\E|u_t(x)|^2=|(\sG_Nu)_t(x)|^2+\lambda^2 \int_0^L\int_0^t p^2_N(t-s,\,x,\,y)\E|\sigma(u_s(y))|^2\d s\,\d y.
\end{equation}
We bound the first term on the right hand side of the above display first. Since $(\sG_Nu)_t(x)$ solves the corresponding deterministic problem, we have that for fixed $t>0$, $(\sG_Nu)_t(x)$ is bounded below by a positive constant depending on $t$.


We now deal with the second term. We will again use \eqref{neulower} as well as the definition of $\sI_t(\lambda)$.

\begin{equation*}
\begin{aligned}
\hskip1.0in\lambda^2 \int_0^L\int_0^t &p^2_N(t-s,\,x,\,y)\E|\sigma(u_s(y))|^2\d s\,\d y\\
&\geq\lambda^2 l_\sigma^2 \int_0^L\int_0^t p_N^2(t-s,\,x,\,y)\E|u_s(y))|^2\d s\,\d y\\
&\geq\lambda^2 l_\sigma^2 \int_0^t \sI_s(\lambda) \int_0^L  p_N^2(t-s,\,x,\,y)\d y\,\d s\\
&\geq\lambda^2 l_\sigma^2 \int_0^t \sI_s(\lambda)p_N(2(t-s),\,x,\,x)\,\d s\\
&\geq\frac{\lambda^2 l_\sigma^2}{\sqrt{4\pi}} \int_0^t \frac{\sI_s(\lambda)}{\sqrt{t-s}}\,\d s.
\end{aligned}
\end{equation*}
Combining the above inequalities, we obtain
\begin{equation*}
\sI_t(\lambda)\geq c_1+\frac{\lambda^2 l_\sigma^2}{\sqrt{4\pi}} \int_0^t \frac{\sI_s(\lambda)}{\sqrt{t-s}}\,\d s.
\end{equation*}
An application of Proposition \ref{lowerbound} yields the lower bound stated in the theorem. We now prove the lower bound.  Our starting point is \eqref{eq:iso-neu}. Finding an upper bound on the first term is straight forward since the initial condition is a bounded function.  For the second term, we need a bit more work.

\begin{equation*}
\begin{aligned}
\hskip1.0in\lambda^2 \int_0^L\int_0^t &p^2_N(t-s,\,x,\,y)\E|\sigma(u_s(y))|^2\d s\,\d y\\
&\leq\lambda^2 L_\sigma^2 \int_0^L\int_0^t p_N^2(t-s,\,x,\,y)\E|u_s(y))|^2\d s\,\d y\\
&\leq\lambda^2 L_\sigma^2 \int_0^t \sS_s(\lambda) \int_0^L  p_N^2(t-s,\,x,\,y)\d y\,\d s\\
&=\lambda^2 L_\sigma^2 \int_0^t \sS_s(\lambda)p_N(2(t-s),\,x,\,x)\,\d s\\
&\leq c_T\frac{\lambda^2 L_\sigma^2}{\sqrt{4\pi}} \int_0^t \frac{\sS_s(\lambda)}{\sqrt{t-s}}\,\d s.
\end{aligned}
\end{equation*}
With this inequality, \eqref{eq:iso-neu} reduces to

\begin{equation*}
\sS_t(\lambda)\leq c_2+\frac{c_T\lambda^2 L_\sigma^2}{\sqrt{4\pi}} \int_0^t \frac{\sS_s(\lambda)}{\sqrt{t-s}}\,\d s.
\end{equation*}
An application of Proposition \ref{upperbound} now yields the desired result.

\end{proof}

\subsection{Proof of Corollary \ref{corr:energy}}
The proof of Corollary \ref{corr:energy} is straightforward.

\begin{proof}
Note that
\begin{eqnarray*}
\|u_t\|^2_{L^2[0,\,L]}\leq \sS^2_t(\lambda) L,
\end{eqnarray*}
from which the upper bound follows. As for the lower bound, we have
\begin{eqnarray*}
\sI^2_t(\lambda) L\leq \|u_t\|^2_{L^2[0,\,L]}.
\end{eqnarray*}
\end{proof}

\begin{small}

\end{small}

\end{document}